\documentclass[10pt]{article}
\usepackage{amsmath,amssymb,amsthm,latexsym}
\usepackage{graphicx}
\renewenvironment{proof}{\noindent{\bf Proof.}}{~~$\Box$}
\theoremstyle{plain}
\newtheorem{thrm}{Theorem}[section]
\newtheorem{lemm}[thrm]{Lemma}
\newtheorem{prot}[thrm]{Proposition}

\newtheorem{corl}[thrm]{Corollary}

\theoremstyle{definition}
\newtheorem{dfnt}[thrm]{Definition}
\newtheorem{remk}[thrm]{Remark}
\newtheorem{exa}[thrm]{Example}

\numberwithin{equation}{section}

\begin{document}
\title{\bf On weak dominance of t-conorms over t-norms
}
\author{
Shuangquan Li\quad Yao Ouyang\thanks{Corresponding author.\quad Email:1760851309@qq.com(S.
Li); oyy@huznu.edu.cn(Y. Ouyang)}\\
\small\it Faculty of Science, Huzhou Normal University, Huzhou, Zhejiang 313000, China}

\date{}
\maketitle
\begin{abstract}
\hspace{4mm}
The weak dominance of aggregation operators, particularly between triangular norms (t-norms) and triangular conorms (t-conorms), has attracted considerable attention in aggregation operator theory. While several characterizations have been obtained for Archimedean and continuous cases, a general criterion for continuous t-conorms over continuous t-norms remains to be fully clarified. In this paper, we provide a complete characterization of a continuous t-conorm weakly dominating a continuous t-norm. We first reduce the problem for ordinal sum operators to that for their single Archimedean components, and then express the weak dominance condition entirely in terms of the additive generators of these components. Our approach covers both strict and nilpotent cases uniformly, and recovers the known results for Archimedean operators as a special case.

{\it Keywords:} Weak dominance; Triangular norms; Triangular conorms; Additive generator
\end{abstract}

\section{Introduction}
The origins of t-norms can be traced back to Menger's seminal paper ``Statistical Metrics'' \cite{M42}. In contrast to classical metric spaces, in which distances are given by nonnegative real numbers, Menger used a class of probability distribution functions on $[0, \infty]$ (denoted by $\Delta^+$) to characterize the distance between elements. When extending the triangle inequality, Menger naturally made use of a class of functions resembling t-norms. The axiomatic definition of t-norms is due to Schweizer and Sklar, who gave the current definition in their papers on statistical metric spaces (subsequently called probabilistic metric spaces); see \cite{SS60,SS61}.

It is well known that the Cartesian product of two metric spaces, endowed with the product metric, is again a metric space. However, for the Cartesian product of two probabilistic metric spaces to also be a probabilistic metric space, an additional condition, called dominance, is required. The dominance relation also plays an important role in the theory of fuzzy relations and aggregation operators \cite{DM98,SBKM06}. Consequently, the dominance between t-norms has attracted considerable attention \cite{S05,S08,SDD09,S84}.

Another closely related notion is weak dominance of aggregation operators. It should be noted that, despite the terminology, weak dominance is not in general a consequence of dominance; nevertheless, if two operators have a common identity element, then dominance entails weak dominance. The weak dominance relation among associative aggregation operators (e.g., t-norms and t-conorms) has been extensively studied.

Alsina et al. \cite{AFS06} characterized the weak dominance of two strict t-norms. Li et al. \cite{LZWL23} investigated the weak dominance between t-norms and t-conorms. They characterized the weak dominance of an Archimedean t-conorm over an Archimedean t-norm,  and obtained some conclusions for the case where the t-conorm and t-norm share the same ordinal sum decomposition. In \cite{LQ24}, Liu and Qin reexamined the weak dominance between t-norms and t-conorms. While they give an elegant characterization of a continuous t-norm weakly dominating another, their characterization of a continuous ordinal sum t-conorm weakly dominating a continuous ordinal sum t-norm relies on another such pair.



In this paper, we continue to investigate the weak dominance of continuous t-conorms over continuous t-norms. We first reduce the weak dominance problem of an ordinal sum t-conorm over an ordinal sum t-norm to that between their single components. Then, we characterize this reduced weak dominance in terms of the additive generators of the corresponding components of the t-conorm and the t-norm. In this way, we obtain a complete characterization of the weak dominance of an ordinal sum t-conorm over an ordinal sum t-norm purely in terms of the additive generators of their components. As a special case, we derive the weak dominance condition for continuous Archimedean t-conorms over continuous Archimedean t-norms, thereby covering the results of Li et al. \cite{LZWL23}. It is worth pointing out that, with our approach, there is no need to distinguish between the strict and nilpotent cases, which simplifies the argument considerably.

\section{Preliminaries}

We recall here the concepts and notations needed in the sequel; see \cite{AFS06,KMP00,SS83} for a comprehensive treatment.

\begin{dfnt}
Let $[a, b]$ be a subinterval of $[0, 1]$. A mapping $T\colon [a, b]^2\to [a, b]$ (respectively, $S\colon [a, b]^2\to [a, b]$) is called a \emph{t-norm} (respectively, \emph{t-conorm}) on $[a, b]$ if it is associative, commutative, nondecreasing in each variable, and satisfies the boundary condition
\[
T(b, x)=x,\, \forall\, x\in[a, b]
\quad\text{(respectively, } S(a, x)=x,\, \forall\, x\in[a, b]\text{)}.
\]
When $[a, b]=[0, 1]$, we simply call $T$ a t-norm and $S$ a t-conorm.
\end{dfnt}

A t-norm $T$ on $[a, b]$ is said to be \emph{continuous} if it is continuous as a two-place function. A continuous t-norm $T$ on $[a, b]$ is called \emph{Archimedean} if $T(x, x)<x$ for all $x\in ]a, b[$. It is well-known that a t-norm $T$ on $[a, b]$ is continuous Archimedean if and only if there is a continuous and strictly decreasing function $t\colon [a, b]\to [0, \infty]$ with $t(b)=0$ such that
\[
T(x, y)=t^{(-1)}(t(x)+t(y)),
\]
where $t^{(-1)}$ is the pseudo-inverse of the \emph{additive generator} $t$ defined by $t^{(-1)}(u)=t^{-1}\bigl(\min\{u, t(a)\}\bigr)$. A continuous Archimedean t-norm $T$ on $[a, b]$ is said to be \emph{strict} if $t(a)=\infty$ and to be \emph{nilpotent} if $t(a)<\infty$.

A t-norm $T$ on $[a, b]$ is said to have \emph{zero divisors} if there exist $x, y\in ]a, b[$ such that $T(x, y)=a$. Clearly, a strict t-norm has no zero divisors, whereas for a nilpotent t-norm every $x\in ]a, b[$ is a zero divisor.

Typical t-norms include the weakest t-norm $T_D$, the strongest t-norm $T_M$, the product t-norm $T_P$, and the Lukasiewicz t-norm $T_L$. They are defined on $[0,1]^2$ as follows:
\[
\begin{aligned}
T_D(x, y) &=
\begin{cases}
\min(x, y) & \text{if } \max(x, y)=1,\\
0          & \text{otherwise,}
\end{cases} \\[4pt]
T_M(x, y) &= \min(x, y), \\[2pt]
T_P(x, y) &= xy, \\[2pt]
T_L(x, y) &= \max(0, x+y-1).
\end{aligned}
\]
Note that $T_M$ is continuous and non-Archimedean, $T_P$ is strict, whereas $T_L$ is nilpotent.

Let $\{]a_i, b_i[\}_{i\in I}$ be a family of pairwise disjoint subintervals of $[0, 1]$, and let $T_i$ be a t-norm on $[a_i,b_i]$ for each $i\in I$. Define $T\colon [0, 1]^2\to [0, 1]$ by
\[
T(x,y) =
\begin{cases}
T_i(x, y) &\text{if~} (x, y)\in [a_i, b_i]^{2} \text{~for some~} i\in I ,\\
\min\{x, y\}          & \text{otherwise.}
\end{cases}
\]
Then $T$ is a t-norm, called the \emph{ordinal sum} of $\{T_i\}_{i\in I}$, and denoted by $T=(\langle a_i,b_i,T_i\rangle)_{i\in I}$. Note that if $I=\emptyset$, then $T=T_M$.

A fundamental result concerning t-norms is the following representation theorem.

\begin{thrm}
For a function $T\colon [0, 1]^2\to [0, 1]$ the following are equivalent:\\
{\rm(i)} $T$ is a continuous t-norm;\\
{\rm(ii)} $T$ can be uniquely represented as an ordinal sum of continuous Archimedean t-norms {\rm(}with the empty ordinal sum understood as $T_M${\rm)}.
\end{thrm}

The dual counterparts of all the above notions can be obtained for t-conorms, since whenever $T$ is a t-norm on $[a, b]$, the function
\[
S(x, y)=b- T(b-x, b-y)
\]
defines a t-conorm on $[a, b]$. For brevity, we merely mention that $S$ is a continuous Archimedean t-conorm on $[a, b]$ if and only if there exists a continuous and strictly increasing function $s\colon [a, b]\to [0, \infty]$ with $s(a)=0$ such that
\[
S(x, y)=s^{(-1)}(s(x)+s(y)),
\]
where $s^{(-1)}$ is the pseudo-inverse of the additive generator $s$, given by $s^{(-1)}(u)=s^{-1}\bigl(\min\{u, s(b)\}\bigr)$.

\section{Main results}
\subsection{Weak dominance: definitions and basic cases}

The weak dominance relation was introduced in \cite{AFS06} for binary operators on $[0, 1]$. Here we are primarily interested in the weak dominance of t-conorms over t-norms.

\begin{dfnt}\cite{AFS06}
Let $T$ be a t-norm and $S$ a t-conorm. We say that $S$ \emph{weakly dominates} $T$, and denoted by $S\gg_w T$, if
\begin{equation}\label{eq:defwd}
S(T(x, y), z)\geq T(S(x, z), y)
\end{equation}
for all $x, y, z\in[0,1]$.
\end{dfnt}

Let us first examine some extremal cases.

\begin{prot}
Let $T$ be a t-norm. Then $S_D\gg_w T$ if and only if $T$ has no zero divisors.
\end{prot}

\begin{proof}
Assume first that $T$ has zero divisors. Then there exist $c\in(0, 1)$ such that $T(c, c)=0$. Fix $z\in ]0, c[$. Since $z>0$,
\[
S_D(T(c, c),z)=S_D(0, z)=z.
\]
On the other hand, $c>0$ and $z>0$, hence $S_D(c, z)=1$, and therefore
\[
T(S_D(c, z), c)=T(1, c)=c.
\]
Because $z<c$, the weak-dominance inequality fails. Thus $S_D\gg_w T$ implies that $T$ has no zero divisors.

Conversely, the case where at least one of $x,y,z$ is zero is immediate from the boundary conditions. For $x,y,z>0$, since $T$ has no zero divisors, we have $T(x, y)>0$. Therefore
\[
S_D(T(x, y), z)=1\ge T(S_D(x, z), y),
\]
so \eqref{eq:defwd} holds. Hence $S_D\gg_w T$.
\end{proof}

\begin{thrm}\label{thrm-SMwdT}
Let $S$ be a t-conorm. Then the following are equivalent:\\
{\rm(i)} $S=S_M$;\\
{\rm(ii)} $S\gg_w T$ for all t-norms $T$.
\end{thrm}

\begin{proof}
The implication (i)$\Rightarrow$(ii) is exactly Proposition 14 in \cite{LZWL23}. For the converse implication, Li et al. proved that if $S\gg_w T$ for all t-norms $T$, then $S(x, y)\in\bigl\{1, \max\{x, y\}\bigr\}$ for all $x, y\in [0, 1]$ (see Proposition 15 in \cite{LZWL23}). Thus, if $S\neq S_M$ then there exists $c\in ]0, 1[$ such that $S(c, c)=1$. Take $x, y\in ]c, 1[$. Then
\[
S(T_D(x, y), c)=S(0, c)=c,
\]
but
\[
T_D(S(x, c), y)=T_D(1, y)=y,
\]
which violates \eqref{eq:defwd}. Hence, (ii)$\Rightarrow$(i).
\end{proof}

Dually, we have the following two results concerning extremal t-norms.

\begin{prot}
Let $S$ be a t-conorm. Then $S\gg_w T_D$ if and only if $S$ has no 1-divisors, i.e. $S(x, y)<1$ for all $x, y<1$.
\end{prot}

\begin{thrm}\label{thrm-TMwdS}
Let $T$ be a t-norm. Then the following are equivalent:\\
{\rm(i)} $T=T_M$;\\
{\rm(ii)} $S\gg_w T$ for all t-conorms $S$.
\end{thrm}

\subsection{Characterization of the continuous case}
In this subsection, we focus on the case where both the t-conorm and the t-norm are continuous, and give a complete characterization of the dominance relation. We need the following two elementary lemmas.

\begin{lemm}\label{lem-autoRegions-1}
Let $T$ be a t-norm. If $T(x, y)=\min\{x, y\}$, then for all t-conorms $S$ and all $z\in[0, 1]$,
\[
S(T(x, y), z)\geq T(S(x, z), y).
\]
\end{lemm}

\begin{proof}
If $x\leq y$, then
\[
S(T(x, y), z)=S(x, z)\geq T(S(x, z), y)
\]
since $T\leq T_M$. If $y<x$, then again by $T\leq T_M$, it holds
\[
S(T(x, y), z)=S(y, z)\geq y\geq T(S(x, z), y).
\]
\end{proof}

\begin{lemm}\label{lem-autoRegions-2}
Let $S$ be a t-conorm. If $S(x, z)=\max\{x, z\}$, then for all t-norms $T$ and all $y\in [0, 1]$
\[
S(T(x, y), z)\geq T(S(x, z), y).
\]
\end{lemm}

\begin{proof}
The proof is analogous to that of Lemma \ref{lem-autoRegions-1}.
\end{proof}

\begin{remk}
From now on, whenever a continuous t-norm or t-conorm is represented as an ordinal sum, we shall always assume, without further explicit mention, that each summand (i.e., each component block) is Archimedean.
\end{remk}

Now, let $S$ be a continuous t-conorm and $T$ a continuous t-norm. If $S=S_M$ or $T=T_M$ then $S\gg_w T$ follows from Theorems \ref{thrm-SMwdT} and \ref{thrm-TMwdS}. So, without loss of generality, we may assume that there exist a family of pairwise disjoint open intervals $\{]a_j, b_j[\}_{j\in J}$ and, for each $j\in J$, an Archimedean t-conorm $S_j$ on $[a_j,b_j]$, such that $S=(\langle a_j, b_j, S_j\rangle)_{j\in J}$; and a family of pairwise disjoint open intervals $\{]c_i, d_i[\}_{i\in I}$ and, for each $i\in I$, an Archimedean t-norm $T_i$ on $[c_i,d_i]$, such that $T=(\langle c_i, d_i, T_i\rangle)_{i\in I}$. In order to verify whether $S\gg_w T$, by Lemmas \ref{lem-autoRegions-1} and \ref{lem-autoRegions-2}, it suffices to verify whether \eqref{eq:defwd} holds for all $i\in I$, $j\in J$, all $(x, y)\in ]c_i, d_i[^2$, and all $(x, z)\in ]a_j, b_j[^2$.

Thus, we introduce the index set
\[
\Lambda=\{(j, i)\in J\times I\mid\, ]a_j, b_j[\cap ]c_i, d_i[\neq\emptyset\}.
\]

For each $(j,i)\in\Lambda$, let $\widehat S_j$ be the single-block ordinal sum obtained by retaining the block $[a_j,b_j]^2$ and replacing all other regions with the maximum operation, and let $\widehat T_i$ be the single-block ordinal sum obtained by retaining the block $[c_i,d_i]^2$ and replacing all other regions with the minimum operation; that is,
\begin{align*}
{\widehat S}_j(x,y) &=
\begin{cases}
S_j(x, y) &\text{if~} (x, y)\in [a_j, b_j]^{2},\\
\max\{x, y\}          & \text{otherwise,}
\end{cases}\\[8pt]
{\widehat T}_i(x,y) &=
\begin{cases}
T_i(x, y) &\text{if~} (x, y)\in [c_i, d_i]^{2},\\
\min\{x, y\}          & \text{otherwise.}
\end{cases}
\end{align*}

We now present a first characterization of $S\gg_w T$ as follows.

\begin{thrm}\label{thrm-firstcharac}
Let $S=(\langle a_j, b_j, S_j\rangle)_{j\in J}$ be a continuous t-conorm, and $T=(\langle c_i, d_i, T_i\rangle)_{i\in I}$ a continuous t-norm. Then $S\gg_w T$ if and only if $\widehat{S}_j\gg_w\widehat{T}_i$ for all $(j, i)\in\Lambda$.
\end{thrm}

\begin{proof}
Suppose $S\gg_w T$. Let $(j, i)\in\Lambda$ and $x, y, z\in [0, 1]$. If $x, y\notin ]c_i, d_i[$, then $\widehat {T}_i(x, y)=\min\{x, y\}$. By Lemma \ref{lem-autoRegions-1}, we have $\widehat{S}_j(\widehat {T}_i(x, y), z)\geq \widehat {T}_i(\widehat {S}_j(x, z), y)$. Similarly, if $x, z\notin ]a_j, b_j[$, then $\widehat{S}_j(x, z)=\max\{x, z\}$. By Lemma \ref{lem-autoRegions-2}, it follows that $\widehat{S}_j(\widehat {T}_i(x, y), z)\geq \widehat {T}_i(\widehat {S}_j(x, z), y)$. If $x, y\in ]c_i, d_i[$ and $x, z\in ]a_j, b_j[$, then
 \[
\widehat{S}_j(\widehat {T}_i(x, y), z)=S(T(x, y), z)\geq T(S(x, z), y)= \widehat {T}_i(\widehat {S}_j(x, z), y),
\]
where the inequality follows from $S\gg_w T$. Thus the required inequality holds for all cases, and consequently $\widehat {S}_j\gg_w \widehat {T}_i$ for all $(j, i)\in\Lambda$.

Conversely, suppose $\widehat {S}_j\gg_w \widehat {T}_i$ for all $(j, i)\in\Lambda$. Let $x, y, z\in [0, 1]$. If $x, y\notin ]c_i, d_i[$ for all $i\in I$, then $T(x, y)=\min\{x, y\}$ and $S(T(x, y), z)\geq T(S(x, z), y)$ follows from Lemma \ref{lem-autoRegions-1}. If $x, z\notin ]a_j, b_j[$ for all $j\in J$, then $S(x, z)=\max\{x, z\}$ and $S(T(x, y), z)\geq T(S(x, z), y)$ follows from Lemma \ref{lem-autoRegions-2}. If $x, y\in ]c_i, d_i[$ for some $i$ and $x, z\in ]a_j, b_j[$ for some $j$, then $(j, i)\in\Lambda$. Thus $\widehat {S}_j\gg_w \widehat {T}_i$, which implies that
\[
S(T(x, y), z)=\widehat{S}_j(\widehat {T}_i(x, y), z)\geq  \widehat {T}_i(\widehat {S}_j(x, z), y)=T(S(x, z), y).
\]
Hence $S\gg_w T$.
\end{proof}

According to Theorem \ref{thrm-firstcharac}, the verification of weak dominance of a continuous t-conorm over a continuous t-norm reduces to checking the corresponding single-block ordinal sum components pairwise. Now, let $\widehat{S}=(\langle a, b, S\rangle)$ and $\widehat{T}=(\langle c, d, T\rangle)$ with $]a, b[\cap ]c, d[\neq\emptyset$, where the components $S$ and $T$ are Archimedean. We will characterize $\widehat{S}\gg_w \widehat{T}$ by means of the additive generators $s$ and $t$ of $S$ and $T$, respectively.

Let $l=\max\{a, c\}$ and $r=\min\{b, d\}$, then $]a, b[\cap ]c, d[=]l, r[$. According to the above discussion, we need only to verify whether $\widehat{S}(\widehat{T}(x, y), z)\geq\widehat{T}(\widehat{S}(x, z), y)$ holds for all $x\in ]l, r[, y\in ]c, d[$ and $z\in ]a, b[$. Moreover, for such $x, y, z$, if $\widehat{S}(\widehat{T}(x, y), z)\geq d$ or $\widehat{T}(\widehat{S}(x, z), y)\leq a$, then the inequality holds automatically. Hence, we introduce the set
\[
E\triangleq\{(x, y, z)\in ]l, r[\times ]c, d[\times ]a, b[\mid \widehat{S}(\widehat{T}(x, y), z)<d\, \mbox{and}\, \widehat{T}(\widehat{S}(x, z), y)>a\}.
\]
Now, we need only to check $\widehat{S}(\widehat{T}(x, y), z)\geq\widehat{T}(\widehat{S}(x, z), y)$ for all $(x, y, z)\in E$.

Let $s$ and $t$ be the additive generators of the Archimedean components $S$ and $T$, respectively. Define $g\colon [s(l), s(b)]\to [0, \infty]$ by
\[
g(\xi) =
\begin{cases}
t(s^{-1}(\xi)) &\text{if~} s(l)\leq\xi\leq s(r)\\
t(r)          & \text{if~} s(r)<\xi\leq s(b)
\end{cases}
\]
and $H\colon [t(r), \infty]\to [0, \infty]$ by
\[
H(\xi) =
\begin{cases}
h(\xi) &\text{if~} t(r)\leq\xi\leq t(l)\\
s(l)          & \text{if~} \xi>t(l),
\end{cases}
\]
where $h=s\circ t^{-1}$.

\begin{lemm}\label{lemm-equivcondition}
$\widehat{S}\gg_w\widehat{T}$ if and only if
\[
\min\{H(g(u)+\eta)+v, s(b)\}\geq H(g(w)+\eta)
\]
holds for all $(x, y, z)\in E$, where $u=s(x), v=s(z), w=\min\{u+v, s(b)\}$ and $\eta=t(y)$.
\end{lemm}

\begin{proof}
First, we prove the following identity for all $(x,y,z)\in E$:
\[
H(g(u)+\eta)=
\begin{cases}
0, &\text{if~} \widehat{T}(x, y)<a\\
s(\widehat{T}(x, y)) &\text{if~} \widehat{T}(x, y)\geq a.
\end{cases}
\]
For the case $l=a>c$, if $\widehat{T}(x, y)\geq a$ then $g(u)+\eta=t(x)+t(y)\leq t(a)=t(l) (<t(c))$, which implies that
\[
H(g(u)+\eta)=h(g(u)+\eta)=s(t^{-1}(t(x)+t(y)))=s(\widehat{T}(x, y));
\]
if $\widehat{T}(x, y)< a$ then $g(u)+\eta=t(x)+t(y)>t(l)$, which implies that $H(g(u)+\eta)=s(l)=s(a)=0$. For the case $l=c$, we have $\widehat{T}(x, y)\geq c=l\geq a$. If $\widehat{T}(x, y)>l$ then $g(u)+\eta=t(x)+t(y)<t(l)$, and again $H(g(u)+\eta)=s(t^{-1}(t(x)+t(y)))=s(\widehat{T}(x, y))$; if $\widehat{T}(x, y)=l$ then $g(u)+\eta=t(x)+t(y)\geq t(l)$, hence
\[
H(g(u)+\eta)=s(l)=s(\widehat{T}(x, y)).
\]
Thus the claimed identity holds.

Second, we claim that, for all $(x,y,z)\in E$,
\[
s(\widehat{S}(\widehat{T}(x, y), z))=\min\{H(g(u)+\eta)+v, s(b)\}.
\]
For the case $\widehat{T}(x, y)<a$, we have $\widehat{S}(\widehat{T}(x, y), z)=\max\{\widehat{T}(x, y), z\}=z$. Hence
\[
s(\widehat{S}(\widehat{T}(x, y), z))=s(z)=v=\min\{v, s(b)\}=\min\{H(g(u)+\eta)+v, s(b)\},
\]
since $\widehat{T}(x, y)<a$ implies $H(g(u)+\eta)=0$. For the case $\widehat{T}(x, y)\geq a$, we have
\begin{align*}
s(\widehat{S}(\widehat{T}(x, y), z)) &=\min\{s(\widehat{T}(x, y))+s(z), s(b)\}\\
                                     &=\min\{H(g(u)+\eta)+v, s(b)\}.
\end{align*}

Moreover, by the definition of $E$, we have $\widehat{T}(\widehat{S}(x, z), y)> a$ for all $(x, y, z)\in E$. Hence
\begin{align*}
s(\widehat{T}(\widehat{S}(x, z), y)) &=H(g(s(S(x, z)))+t(y))\\
                                     &=H(g(\min\{s(x)+s(z), s(b)\})+\eta)\\
                                     &=H(g(\min\{u+v, s(b)\})+\eta)=H(g(w)+\eta).
\end{align*}

Thus, $\widehat{S}\gg_w\widehat{T}$ if and only if $\widehat{S}(\widehat{T}(x, y), z)\geq \widehat{T}(\widehat{S}(x, z), y)$ holds for all $(x, y, z)\in E$, and the latter inequality holds if and only if $\min\{H(g(u)+\eta)+v, s(b)\}\geq H(g(w)+\eta)$.
\end{proof}

\begin{thrm}\label{thrm-charac-2}
$\widehat{S}\gg_w\widehat{T}$ if and only if $h=s\circ t^{-1}$ is convex on $[t(r), t(l)]$.
\end{thrm}

\begin{proof}
Note that $h$ is convex on $[t(r), t(l)]$ if and only if $H$ is convex, since $H|_{[t(r), t(l)]}=h$ and $H$ is continuous, non-increasing, and constant on $[t(l), \infty]$.

Suppose $H$ is convex. Then for all $(x, y, z)\in E$ we have
\begin{equation}\label{eq-convexofH}
H(g(w)+\eta)-H(g(u)+\eta)\leq H(g(w))-H(g(u))
\end{equation}
since $g(w)\leq g(u)$, where $u=s(x), \eta=t(y)$, $w=\min\{u+v, s(b)\}$ and $v=s(z)$. Since $H(g(w))=H(g(\min\{u+v, s(b)\}))=\min\{u+v, s(r)\}$ and $H(g(u))=u$, \eqref{eq-convexofH} implies
\[
H(g(w)+\eta)\leq H(g(u)+\eta)+\min\{u+v, s(r)\}-u\leq H(g(u)+\eta)+v.
\]
Moreover, $H(g(w)+\eta)\leq h(t(r))=s(r)\leq s(b)$. Thus,
\[
\min\{H(g(u)+\eta)+v, s(b)\}\geq H(g(w)+\eta)
\]
which implies $\widehat{S}\gg_w\widehat{T}$ by Lemma \ref{lemm-equivcondition}.

Conversely, suppose $\widehat{S}\gg_w\widehat{T}$. We need to show that $h$ is convex on $[t(r), t(l)]$. Since $h$ is continuous, it suffices to show that $h$ is convex on $]t(r), t(l)[$; equivalently,
\begin{equation}\label{eq-convexconofh}
h(\beta+\eta)-h(\alpha+\eta)\geq h(\beta)-h(\alpha)
\end{equation}
for all $\alpha, \beta, \eta$ satisfying $t(r)<\alpha<\beta$ and $0<\eta<t(l)-\beta$. Since $0<\eta<t(l)-\beta$, there exists $y\in ]l, d[\subseteq ]c, d[$ such that $t(y)=\eta$. Denote $u=h(\beta), w=h(\alpha), v=w-u$ and $x=s^{-1}(u), z=s^{-1}(v)$. Then we have $x\in ]l, r[$ and $z\in ]a, r[\subseteq ]a, b[$. Moreover,
\begin{align*}
\widehat{S}(\widehat{T}(x, y), z) &=s^{(-1)}(s(\widehat{T}(x, y))+s(z))\\
                                  &=s^{(-1)}(h(\beta+\eta)+w-u)\\
                                  &\geq s^{(-1)}(h(t(l))+w-u)\\
                                  &=s^{(-1)}(s(l)+w-u)>l\geq a
\end{align*}
and
\begin{align*}
\widehat{T}(\widehat{S}(x, z), y) &=t^{(-1)}(t(\widehat{S}(x, z))+t(y))\\
                                  &=t^{(-1)}(h(\beta+\eta)+w-u)\\
                                  &=t^{(-1)}(\alpha+\eta)<r\leq b.
\end{align*}
Hence, $(x, y, z)\in E$. Since $\widehat{S}\gg_w\widehat{T}$, it follows from Lemma \ref{lemm-equivcondition} that
\[
H(g(u)+\eta)+v\geq H(g(w)+\eta).
\]
Since $t(r)<\alpha<\beta<\beta+\eta<t(l)$, this inequality is exactly \eqref{eq-convexconofh}. Thus, $h=s\circ t^{-1}$ is convex on $[t(r), t(l)]$.
\end{proof}

Our main theorem is stated as follows.

\begin{thrm}\label{thrm-Main}
Let $S=(\langle a_j, b_j, S_j\rangle)_{j\in J}$ be a continuous t-conorm, and $T=(\langle c_i, d_i, T_i\rangle)_{i\in I}$ a continuous t-norm. Then $S\gg_w T$ if and only if $h_{ji}=s_j\circ t^{-1}_i$ is convex on $[t_i(r_{ji}), t_i(l_{ji})]$ for all $(j, i)\in\Lambda$, where $s_j$ and $t_i$ are additive generators of $S_j$ and $T_i$ respectively, and $l_{ji}=\max\{a_j, c_i\}$ and $r_{ji}=\min\{b_j, d_i\}$.
\end{thrm}

\begin{proof}
It follows from Theorems \ref{thrm-firstcharac} and \ref{thrm-charac-2}.
\end{proof}

\begin{exa}
Let $S=(\langle\frac{1}{4}, \frac{3}{4}, S_1\rangle)$ and $T=(\langle 0, \frac{1}{2}, T_1\rangle, \langle\frac{1}{2}, 1, T_2\rangle)$, where
\[
S_1(x, y)=\frac{1}{4}+\frac{1}{2}\sqrt{\min\!\left\{\left(2x-\frac{1}{2}\right)^2+\left(2y-\frac{1}{2}\right)^2,1\right\}},\quad x, y\in [\frac{1}{4}, \frac{3}{4}],
\]
and
\begin{align*}
T_1(x, y) &=2xy \quad(x, y\in [0, \frac{1}{2}])\\
T_2(x, y) &=\max\{\frac{1}{2}, x+y-1\} \quad(x, y\in [\frac{1}{2}, 1]).
\end{align*}
Then the corresponding additive generators are
\begin{align*}
s_1(x) &= \left(2x-\frac12\right)^2, \quad x\in [\frac{1}{4}, \frac{3}{4}]\\
t_1(x) &=-\ln(2x), \quad x\in [0, \frac{1}{2}]\\
t_2(x) &=2(1-x), \quad x\in [\frac{1}{2}, 1],
\end{align*}
respectively. The composition functions $h_{11}=s_1\circ t^{-1}_1$ and $h_{12}=s_1\circ t^{-1}_2$ are respectively given by
\begin{align*}
h_{11}(x) &=(\exp(-x)-\frac{1}{2})^2, \quad x\in [0, \infty]\\
h_{12}(x) &=(\frac{3}{2}-x)^2, \quad x\in [0, 1].
\end{align*}
So $h_{11}$ is convex on $[t_1(r_{11}), t_1(l_{11})]=[0, \ln 2]$ and $h_{12}$ is convex (hence also on $[t_2(r_{12}), t_2(l_{12})]$). Therefore, by Theorem \ref{thrm-Main}, we have $S\gg_w T$.

It is worth noting that $h_{11}$ fails to be convex for $x>2\ln 2$, so it is not globally convex on its full domain $[0, \infty]$. However, the theorem only requires convexity on the specific local interval arising from the ordinal-sum construction; the global behaviour is irrelevant for the conclusion.
\end{exa}

\begin{exa}
Now, keeping $S$ and $T_2$ fixed, we change only $T_1$ to
\[
T_1(x, y)=\sqrt[4]{\max\!\left\{x^4+y^4-\frac{1}{16},0\right\}},\quad x, y\in [0, \frac{1}{2}].
\]
Its additive generator is $t_1(x)=1- 16x^4$. Hence $[t_{11}(r_{11}), t_11(l_{11})]=[0, \frac{15}{16}]$. Since
\[
h_{11}(x)=s_1\circ t^{-1}_1(x)=\Bigl((1-x)^{\frac{1}{4}}-\frac{1}{2}\Bigr)^2
\]
is not convex on $[0, \frac{15}{16}]$, we conclude that $S\not\gg_w T$. In fact for $x=\frac{9}{20}, y=\frac{17}{40}$ and $z=\frac{2}{5}$, the inequality $S(T(x, y), z)\geq T(S(x, z), y)$ fails.
\end{exa}

Now, we can derive several special cases.

\begin{corl}
Let $S=(\langle a_j, b_j, S_j\rangle)_{j\in J}$ be a continuous t-conorm, and $T$ a continuous Archimedean t-norm. Then $S\gg_w T$ if and only if $h_{j}=s_j\circ t^{-1}$ is convex on $[t(b_{j}), t(a_{j})]$ for all $j$, where $s_j$ and $t$ are additive generators of $S_j$ and $T$, respectively.
\end{corl}

\begin{corl}
Let $S$ be a continuous Archimedean t-conorm and $T=(\langle c_i, d_i, T_i\rangle)_{i\in I}$ a continuous t-norm. Then $S\gg_w T$ if and only if $h_{i}=s\circ t^{-1}_i$ is convex on $[0, t_i(c_{i})]$ for all $i$, where $s$ and $t_i$ are additive generators of $S$ and $T_i$, respectively.
\end{corl}

\begin{corl}{\rm\cite{LZWL23}}
Let $S$ be a continuous Archimedean t-conorm and $T$ a continuous Archimedean t-norm. Then $S\gg_w T$ if and only if $h=s\circ t^{-1}$ is convex on $[0, t(0)]$, where $s$ and $t$ are additive generators of $S$ and $T$, respectively.
\end{corl}

\section{Conclusion}
In this paper, we have established a complete characterization of the weak dominance of a continuous t-conorm $S=(\langle a_j, b_j, S_j\rangle)_{j\in J}$ over a continuous t-norm $T=(\langle c_i, d_i, T_i\rangle)_{i\in I}$. The main result (Theorem \ref{thrm-Main}) states that the weak dominance holds if and only if, for every pair of overlapping ordinal sum components, the composition function $h_{ji}=s_j\circ t^{-1}_i$ is convex on $[t_i(r_{ji}), t_i(l_{ji})]$, where $s_j$ and $t_i$ are additive generators of $S_j$ and $T_i$ respectively, and $[l_{ji}, r_{ji}]=[a_j, b_j]\cap [c_i, d_i]$. This characterization is obtained by first reducing the general ordinal sum problem to the single-block case (Theorem \ref{thrm-firstcharac}), and then proving that the weak dominance of such single-block operators is equivalent to the convexity of the corresponding generator composition (Theorem \ref{thrm-charac-2}).

As consequences, we derive explicit criteria for the cases where either the t-conorm or the t-norm is Archimedean, and also recover the previous result of Li et al. \cite{LZWL23} for continuous Archimedean t-conorms over continuous Archimedean t-norms as a corollary. A notable feature of our method is that it treats strict and nilpotent generators in a unified manner, thereby simplifying the proof. We hope that the method employed here will be useful for further investigations of weak dominance among more general classes of aggregation operators.


\end{document}